\documentclass{amsart}

\usepackage{amsmath,amsthm,amssymb}

\newtheorem{theorem}{Theorem}[section]
\newtheorem{lemma}[theorem]{Lemma}
\newtheorem{proposition}[theorem]{Proposition}

\newcommand{\fn}[1]{\mathrm{#1}}
\newcommand{\na}[1]{\mathsf{#1}}
\newcommand{\ax}[1]{(\mathsf{#1})}
\newcommand{\mdl}[1]{\mathcal{#1}}

\newcommand{\NN}{\mathbb{N}}

\newcommand{\st}{\; : \;}
\newcommand{\concat}{\mathord{\hat{\;}}}

\newcommand{\ex}[1]{\exists #1 \;} % exists x ...
\newcommand{\fa}[1]{\forall #1 \;} % forall x ...
\newcommand{\lam}[1]{\lambda #1 \;} % lambda x ...

\newcommand{\len}{\fn{length}}

\newcommand{\overbar}[1]{\mkern 3mu\overline{\mkern-3mu#1\mkern-1mu}\mkern 1mu}

\title{Metastable convergence theorems}

\author{Jeremy Avigad, Edward Dean, and Jason Rute}


\thanks{Work by the first and third authors has been partially supported by NSF grant DMS-1068829.}

\begin{document}

\begin{abstract}
The dominated convergence theorem implies that if $(f_n)$ is a sequence of functions on a probability space taking values in the interval $[0,1]$, and $(f_n)$ converges pointwise a.e., then $(\int f_n)$ converges to the integral of the pointwise limit. Tao \cite{tao:08} has proved a quantitative version of this theorem: given a uniform bound on the rates of metastable convergence in the hypothesis, there is a bound on the rate of metastable convergence in the conclusion that is independent of the sequence $(f_n)$ and the underlying space. We prove a slight strengthening of Tao's theorem which, moreover, provides an explicit description of the second bound in terms of the first. Specifically, we show that when the first bound is given by a continuous functional, the bound in the conclusion can be computed by a recursion along the tree of unsecured sequences. We also establish a quantitative version of Egorov's theorem, and introduce a new mode of convergence related to these notions. 
\end{abstract}

\maketitle

\section{Introduction}
\label{introduction:section}

If $(a_n)$ is a nondecreasing sequence of real numbers in the interval $[0,1]$, then $(a_n)$ converges, and hence is Cauchy. Say that $r(\varepsilon)$ is a \emph{bound on the rate of convergence} of $(a_n)$ if for every $\varepsilon > 0$, $|a_n - a_{n'}| < \varepsilon$ whenever $n$ and $n'$ are greater than or equal to $r(\varepsilon)$. In general, one cannot compute a bound on the rate of convergence from the sequence itself: such a bound is not even continuous in the data, since the sequence $(a_n)$ can start out looking like a constant sequence of $0$'s and then increase to $1$ unpredictably.

But suppose that instead of a bound on the rate of convergence, we fix a function $F : \NN \to \NN$ and ask for an $m$ such that $|a_n - a_{n'}| < \varepsilon$ for every $n$ and $n'$ in the interval $[m,F(m)]$. Since the sequence $(a_n)$ cannot increase by $\varepsilon$ more than $\lceil 1 / \varepsilon \rceil$ times, at least one element of the sequence $0, F(0), F(F(0)), \ldots, F^{\lceil 1 / \varepsilon \rceil + 1}(0)$ has the desired property. Hence there is always such a value of $m$ less than or equal to $F^{\lceil 1 / \varepsilon \rceil + 1}(0)$.

Now notice that not only is this bound on $m$ easily computable from $F$ and a rational $\varepsilon > 0$, but it is, moreover, entirely independent of the sequence $(a_n)$. What has happened is that we have replaced the assertion
\[
 \fa {\varepsilon > 0} \ex m \fa {n, n' \geq m} |a_n - a_{n'}| < \varepsilon
\]
by a ``metastable'' version,
\[
 \fa {\varepsilon > 0, F} \ex m \fa {n, n' \in [m,F(m)]} |a_n - a_{n'}| < \varepsilon.
\]
The two statements are logically equivalent: an $m$ as in the first statement is sufficient for any $F$ in the second, and, conversely, if the first statement were false for some $\varepsilon > 0$ then for every $m$ we could define $F(m)$ to return a value large enough so that $[m,F(m)]$ includes a rogue pair $n, n'$. But whereas one cannot compute a bound on the $m$ in the first statement from $\varepsilon$ and $(a_n)$, one can easily compute a bound on the second $m$ that depends only on $\varepsilon$ and $F$. 

If $(a_n)$ is any sequence, say that $M(F)$ is a \emph{bound on the $\varepsilon$-metastable convergence of $(a_n)$} if the following holds:
\begin{quote}
For every function $F : \NN \to \NN$ there is an $m \leq M(F)$ such that for every $n,n' \in [m,F(m)]$, $|a_n - a_{n'}| < \varepsilon$. 
 \end{quote}
Then what we have observed amounts to the following:
\begin{itemize}
 \item There is a bound on the $\varepsilon$-metastable convergence of $(a_n)$ if and only if there is an $m$ such that $|a_n - a_{n'}| < \varepsilon$ for all $n, n' \geq m$. Hence, a sequence $(a_n)$ is Cauchy if and only if there is a bound on the $\varepsilon$-metastable convergence of $(a_n)$ for every $\varepsilon > 0$.
 \item For every $\varepsilon > 0$ the function $M(F) = F^{\lceil 1 / \varepsilon \rceil + 1}(0)$ is a bound on the $\varepsilon$-metastable convergence of any nondecreasing sequence $(a_n)$ of elements of the real interval $[0,1]$.
\end{itemize}
Thus there is a sense in which the second statement provides a quantitative, uniform version of the original convergence theorem.

This transformation is an instance of Kreisel's ``no-counterexample'' interpretation \cite{kreisel:51,kreisel:59}, which is, in turn, a special case of G\"odel's \emph{Dialectica} interpretation \cite{avigad:feferman:98,goedel:58,kohlenbach:08}. The particular example above is discussed by Kreisel \cite[page 49]{kreisel:52}. Variations on this idea have played a role in the Green-Tao proof \cite{green:tao:08} that there are arbitrarily long arithmetic progressions in the primes, and in Tao's proof \cite{tao:08} of the convergence of certain diagonal averages in ergodic theory. In these instances the Kreiselian trick takes the form of an ``energy incrementation argument''; see also \cite{tao:06} and \cite[Sections 1.3--1.4]{tao:08b}. The Birkhoff and von Neumann ergodic theorems and generalizations have also been analyzed in these terms \cite{avigad:et:al:10,kohlenbach:leustean:09,kohlenbach:10,kohlenbach:unp:e}. 

Here we are concerned with measure-theoretic facts such as the dominated convergence theorem, which relate one mode of convergence to another. Inspired by Tao \cite{tao:08}, our goal will be to show that from a suitable metastable bound on the first type of convergence, one can obtain a suitable metastable bound on the second; and that, moreover, the passage from the first to the second is uniform in the remaining data. 

For example, if $(f_n)$ is a sequence of measurable functions on a measure space $\mdl X = (X, \mdl B, \mu)$, then $(f_n)$ is said to converge \emph{almost uniformly} if for every $\lambda > 0$, there is a set $A$ with measure at most $\lambda$ such that $(f_n(x))$ converges uniformly for $x \not\in A$. This is equivalent to saying that for every $\lambda > 0$ and $\varepsilon > 0$ there is an $m$ such that $\mu(\{x \st \ex {n, n' \geq m} |f_n(x) - f_{n'}(x)| \geq \varepsilon\}) < \lambda$, since for a fixed $\lambda' > 0$ we can choose a sequence $(\varepsilon_i)$ decreasing to $0$ and then, for each $\varepsilon_i$, apply this last statement with $\lambda = \lambda'/2^{i+1}$. Thus the fact that $f_n$ converges almost uniformly can be expressed as follows:
\begin{equation}
\label{au:eq}
\fa{\lambda > 0, \varepsilon > 0} \ex m \mu(\{ x \st  \ex{n, n' \geq m} |f_n(x) - f_{n'}(x)| \geq \varepsilon \}) < \lambda.
\tag{$\mathsf{AU}$}
\end{equation}
By manipulations similar to the ones above, (\ref{au:eq}) has the following metastable equivalent:
\begin{equation}
\label{mau:eq}
\fa{\lambda > 0, \varepsilon > 0, F} \ex m \mu(\{ x \st \ex {n, n' \in [m,F(m)]} |f_n(x) - f_{n'}(x)| \geq \varepsilon\}) < \lambda. 
\tag{$\mathsf{AU^*}$}
\end{equation}
As above, say that $M(F)$ is a \emph{bound on the $\lambda$-uniform $\varepsilon$-metastable convergence of $(f_n)$} if the following holds:
\begin{quote}
 For every $F$, there is an $m \leq M(F)$ such that 
\[
\mu(\{x \st \ex {n, n' \in [m,F(m)]} |f_n(x) - f_{n'}(x)| \geq \varepsilon\}) < \lambda.
\] 
\end{quote}
In other words, fixing $\lambda$ and $\varepsilon$, $M(F)$ provides a bound on a value of $m$ asserted to exist by (\ref{mau:eq}).

Egorov's theorem asserts that if $\mdl X$ is a probability space and $(f_n)$ converges pointwise almost everywhere, then it converges almost uniformly. In Section~\ref{convergence:section}, we obtain the following quantitative version. Say that $M(F)$ is a \emph{$\lambda$-uniform bound for the $\varepsilon$-metastable pointwise convergence of $(f_n)$} if the following holds:
\begin{quote}
  For every $F : \NN \to \NN$,  
\[
\mu(\{x \st \fa {m \leq M(F)} \ex {n, n' \in [m,F(m)]} |f_n(x) - f_{n'}(x)| \geq \varepsilon\}) < \lambda.
\]
\end{quote}
In other words, for every $F$, $M(F)$ provides a uniform $\varepsilon$-metastable bound for the convergence of each sequence $(f_n(x))$ outside a set of measure at most $\lambda$. Compare this to the previous definition: if $M(F)$ is a bound on the $\lambda$-uniform $\varepsilon$-metastable convergence of $(f_n)$, then $M(F)$ provides a bound on a \emph{single $m$} that works outside a set of measure at most $\lambda$. With this terminology in place, we can state our quantitative version of Egorov's theorem: given $\varepsilon > 0$, $\lambda > \lambda' > 0$, and a $\lambda'$-uniform bound $M_1(F)$ on the $\varepsilon$-metastable pointwise convergence of $(f_n)$, there is a bound $M_2(F)$ on the $\lambda$-uniform $\varepsilon$-metastable convergence of $(f_n)$; and moreover $M_2(F)$ depends only on $\varepsilon$, $\lambda$, $\lambda'$, and $M_1(F)$, and not on the underlying probability space or the sequence $(f_n)$. In fact, we provide an explicit description of $M_2(F)$ in terms of this data, and explicit bounds on the complexity of $M_2$ when $M_1$ is a computable functional that can be defined using G\"odel's schema of primitive recursion in the finite types. The proof relies on a combinatorial lemma, presented in Section~\ref{combinatorial:section}, whose proof can be veiwed as an energy incrementation argument that is iterated along a well-founded tree.

It is easy to show that if $(f_n)$ is a sequence of functions taking values in $[0,1]$ and $(f_n)$ converges almost uniformly, then the sequence $(\int f_n)$ converges. Thus the dominated convergence theorem follows easily from Egorov's theorem in the case where $\mdl X$ is a probability space and the sequence $(f_n)$ is dominated by a constant function. In a similar way, we show in Section~\ref{convergence:section} that our quantitative version of Egorov's theorem implies a quantiative version of the dominated convergence theorem, a mild strengthening of a Theorem~A.2 of Tao \cite{tao:08}, again with an explicit description of the computation of one metastable bound from the other.

The notion of a $\lambda$-uniform bound on the $\varepsilon$-metastable pointwise convergence of a sequence gives rise to a new mode of convergence that sits properly between pointwise convergence and almost uniform convergence. In Section~\ref{new:convergence:section}, we explore the relationships between these notions.

We are grateful to Ulrich Kohlenbach and Paulo Oliva for advice and suggestions.

\section{A combinatorial fact}
\label{combinatorial:section}

This section is devoted to establishing a key combinatorial fact that underlies our quantitative convergence theorems. As a warmup, consider the following:

\begin{proposition}
\label{warmup:prop}
Let $(A_n)$ be a sequence of measurable subsets of a probability space $\mdl X = (X, \mdl B, \mu)$. Then the following are equivalent:
\begin{enumerate}
 \item There is an $M$ such that $\mu(\bigcup_{n \geq M} A_n) < \lambda$.
 \item There is an $M$ such that for every function $F(m)$, 
\[
 \mu \left(\bigcap_{m \leq M} \bigcup_{n \in [m,F(m)]} A_n \right) < \lambda.
\]
 \item There is a $\lambda' < \lambda$ such that for every $F$ there is an $M$ such that 
\[
 \mu \left(\bigcap_{m \leq M} \bigcup_{n \in [m,F(m)]} A_n \right) < \lambda'.
\]
\end{enumerate}
\end{proposition}

\begin{proof}
 (1) clearly implies (2) because 
\[
\bigcap_{m \leq M} \bigcup_{n \in [m,F(m)]} A_n \subseteq \bigcap_{m \leq M} \bigcup_{n \geq m} A_n = \bigcup_{n \geq M} A_n,
\] 
and (2) clearly implies (3). To show (3) implies (1), fix $\lambda > \lambda' > 0$ and for each $m$, let $F(m)$ be large enough so that 
\[
\mu\left(\bigcup_{n \geq m} A_n \setminus \bigcup_{n \in [m,F(m)]} A_n\right) < (\lambda - \lambda') / 2^{m+1}.
\]
By hypothesis, for this $F$, there is an $M$ such that $\mu(\bigcap_{m \leq M} \bigcup_{n \in [m,F(m)]} A_n) < \lambda'$. 
Then
\begin{multline*}
 \bigcup_{n \geq M} A_n = \bigcap_{m \leq M} \bigcup_{n \geq m} A_n \subseteq \\ \left(\bigcap_{m \leq M} \bigcup_{n \in [m,F(m)]} A_n\right) \cup \bigcup_{m \leq M} \left(\bigcup_{n \geq m} A_n \setminus \bigcup_{n \in [m,F(m)]} A_n\right),
\end{multline*}
whose measure is at most $\lambda' + \sum_{m \leq M} (\lambda - \lambda') / 2^{m+1} < \lambda$. Hence $\mu (\bigcup_{n \geq M} A_n) < \lambda$, as required.
\end{proof}

In particular, if (3) holds, there is an $n$ such that $\mu(A_n) < \lambda$. Now suppose we are given a functional $M(F)$ witnessing (3). The main result of this section, Theorem~\ref{combinatorial:thm}, shows that there is a bound on $n$ that depends only on $M(F)$, $\lambda$, and $\lambda'$. In particular, the bound is independent of $\mdl X$ and the sequence $(A_n)$. 
\begin{theorem}
\label{combinatorial:thm}
For every functional $M(F)$ and $\lambda > \lambda' > 0$, there is a value $M'$ with the following property. Suppose $(A_n)$ is a sequence of measurable subsets of a probability space $\mdl X$ with the property that for every function $F$,
\[
\mu\left(\bigcap_{m \leq M(F)} \bigcup_{n \in [m,F(m)]} A_n\right) < \lambda'.
\]
Then there is an $n \leq M'$ such that $\mu(A_n) < \lambda$.
\end{theorem}

A functional $M$ is said to be \emph{continuous} if the value of $M(F)$ depends on only finitely many values of $F$. Say that two functions $F$ and $F'$ \emph{agree up to $k$} if $F(j) = F'(j)$ for every $j \leq k$. If $M$ is continuous, a functional $k(F)$ with the property that $M(F) = M(F')$ whenever $F$ and $F'$ agree up to $k(F)$ is said to be a \emph{modulus of continuity} for $M$. 

The next lemma shows that, without loss of generality, we can assume the functional $M$ in the hypothesis of Theorem~\ref{combinatorial:thm} is continuous, because one can always replace it by a suitable continuous version, $\overbar{M}$.

\begin{lemma}
\label{combinatorial:lemma}
Given any functional $M$, there is a continuous functional $\overbar{M}$ with the following property: for every $F$, there is an $F'$ such that $\overbar{M}(F) = M(F')$ and $F$ and $F'$ agree up to $\overbar{M}(F)$.
\end{lemma}

\begin{proof}
Given $M$, define
\[
\overbar{M}(F) = \min \{ M(F') \st \mbox{$F$ and $F'$ agree up to $M(F')$} \}.
\]
The last set is nonempty since it contains $M(F)$ itself. Clearly $\overbar{M}(F)$ satisfies the stated condition, so we only need to show that $\overbar{M}$ is continuous. 

In fact, we claim that $\overbar{M}$ is its own modulus of continuity. To see this, suppose $F$ and $F''$ agree up to $\overbar{M}(F)$. We need to show $\overbar{M}(F) = \overbar{M}(F'')$. By the definition of $\overbar{M}$, there is an $F'$ such that $\overbar{M}(F) = M(F')$ and $F$ and $F'$ agree up to $M(F')$. But then $F''$ and $F$ agree up to $M(F')$, and so $\overbar{M}(F'') \leq M(F') = \overbar{M}(F)$.

Since $F$ and $F''$ agree up to $\overbar{M}(F)$, \emph{a fortiori}, they agree up to $\overbar{M}(F'')$. But now the symmetric argument shows that $\overbar{M}(F) \leq \overbar{M}(F'')$. So $\overbar{M}(F) = \overbar{M}(F'')$.
\end{proof}

The condition on $\overbar{M}$ imposed by Lemma~\ref{combinatorial:lemma} ensures that any sequence $(A_n)$ of a measure space $\mdl X$ satisfying 
\[
\fa {F'} \mu\left(\bigcap_{m \leq M(F')} \bigcup_{n \in [m,F'(m)]} A_n\right) < \lambda'
\]
also satisfies
\[
\fa {F} \mu\left(\bigcap_{m \leq \overbar{M}(F)} \bigcup_{n \in [m,F(m)]} A_n\right) < \lambda'
\]
and so it suffices to prove Theorem~\ref{combinatorial:thm} for $\overbar{M}$ in place of $M$. By similar machinations, we could arrange that $\overbar{M}(F) \leq \overbar{M}(G)$ whenever $F$ is pointwise less than or equal to $G$, and that $\overbar{M}$ is determined by the values it takes on nondecreasing $F$. However, we will not need these additional conveniences below. 

Notice that the passage from $M$ to $\overbar{M}$ is noneffective; in general it will not be possible to ``compute'' $\overbar{M}(F)$ from descriptions of $M$ and $F$. We will show, however, that in the case where $M$ is continuous, the $M'$ in the conclusion of Theorem~\ref{combinatorial:thm} \emph{can} be computed from a suitable description of $M$.

To explain our algorithm, we need to establish some background involving computation on well-founded trees. If $\sigma$ is a finite sequence of natural numbers, we index the elements starting with $0$ so that $\sigma = (\sigma_0, \ldots, \sigma_{\len(\sigma) - 1})$, and write $\sigma \concat n$ to denote the sequence extending $\sigma$ with an additional element $n$. If $\tau$ is another finite sequence of natural numbers, write $\sigma \subseteq \tau$ to indicate that $\sigma$ is an initial segment of $\tau$.  By a \emph{tree on $\NN$}, we mean a set $T$ of finite sequences of natural numbers that is closed under initial segments. Think of the empty sequence, $()$, as denoting the root, and the elements $\sigma\concat n$ as being the children of $\sigma$ in the tree. 

Identify functions $F$ from $\NN$ to $\NN$ with infinite sequences, and write $\sigma \subset F$ if $\sigma$ is an initial segment of $F$. A tree $T$ on $\NN$ is said to be \emph{well-founded} if it has no infinite branch, which is to say, for every function $F$ there is a $\sigma \subset F$ such that $\sigma$ is not in the tree. One can always carry out a proof by induction on a well-founded tree: if $P_\sigma$ is any property that holds outside a tree $T$ and moreover has the property that $P_\sigma$ holds whenever $P_{\sigma \concat n}$ holds for every $n$, then $P_\sigma$ holds for every $\sigma$; otherwise, one could successively extend a counterexample $\sigma$ to build an infinite branch $F$ that never leaves the tree. By the same token, one can define a function on finite sequences of natural numbers by a schema of recursion:
\[
 G(\sigma) = \left\{
  \begin{array}{ll}
    H(\sigma) & \mbox{if $\sigma$ is not in $T$} \\
    K(\sigma, \lam {n.} G(\sigma \concat n)) & \mbox{otherwise}
  \end{array}
 \right.
\]
where $\lam {n.} G(\sigma \concat n)$ denotes the function which maps $n$ to $G(\sigma \concat n)$. Using induction on $T$, one can show that $G$ is well-defined. Moreover, if $T$ and the functions $H$ and $K$ are computable, so is $G$. For example, the computation of $G$ on the empty string requires recursive calls to $G((n))$, for various $n$; these, in turn, require recursive calls to $G((n,n'))$, for various $n'$, and so on. The well-foundedness of $T$ guarantees that every branch of the computation terminates.

Now suppose $M(F)$ is a continuous functional. Say that a finite sequence $\sigma$ is \emph{unsecured} if there are $F_1, F_2$ extending $\sigma$ such that $M(F_1) \neq M(F_2)$. In words, $\sigma$ is unsecured if it does not provide sufficient information about a function $F$ to determine the value of $M$. Let $T = \{ \sigma \st \mbox{$\sigma$ is unsecured} \}$. Then it is not hard to see that $T$ is a tree, and the continuity of $M$ implies it is well-founded.

Suppose moreover that $k(F)$ is a modulus for $M$. For any finite sequence $\sigma$ of natural numbers, use $\hat \sigma$ to denote the function
\[
 \hat \sigma(n) = \left\{
  \begin{array}{ll}
  \sigma_n & \mbox{if $n < \len(\sigma)$} \\
  0 & \mbox{otherwise.}
  \end{array}
 \right.
\]
One can check that the set $T' = \{ \sigma \st \fa{\tau \subseteq \sigma} k(\hat \tau) \geq \fn{length}(\tau) \}$ is again a well-founded tree that includes $T$. In the next proof, given a continuous functional $M$, we will define a function $N(\sigma)$ by recursion on any well-founded tree that includes the tree of sequences that are unsecured for $M$. When this tree is given by a modulus of continuity, $k(F)$, as above, this amounts to the principle of \emph{bar recursion}, due to Spector \cite{spector:62} (see also \cite{avigad:feferman:98,kohlenbach:08}).

We now turn to the proof of Theorem~\ref{combinatorial:thm}.

\begin{proof}
By Lemma~\ref{combinatorial:lemma}, we can assume without loss of generality that $M$ is continuous. Fix $\lambda > \lambda' > 0$, and let $T$ be any well-founded tree that includes all the sequences that are unsecured for $M$. We will define a function $N(\sigma)$ by recursion on $T$, and simultaneously show, by induction on $T$, that $N(\sigma)$ satisfies the following property, $P_\sigma$, for every $\sigma$: whenever $\mdl X$  and $(A_n)$ satisfy
\begin{equation}
\tag{$Q_\sigma$}
\fa {F \supset \sigma} \mu\left(\bigcap_{m \leq M(F)} \bigcup_{n \in [m,F(m)]} A_n\right) < \lambda' \\
\end{equation}
and
\begin{equation}
\tag{$R_\sigma$}
\fa {m < \len(\sigma)} \mu \left(\bigcup_{n \in [m,N(\sigma)]} A_n \setminus \bigcup_{n \in [m,\sigma_m]} A_n\right) \leq (\lambda - \lambda') / 2^{m+1},
\end{equation}
there is an $n \leq N(\sigma)$ such that $\mu(A_n) < \lambda$. In that case, $N(())$ is the desired bound, since $Q_{()}$ is the desired hypothesis, and $R_{()}$ is vacuously true.

In the base case, suppose $\sigma$ is not in $T$, and hence secured for $M$. Define $N(\sigma) = M(\hat \sigma)$. To see that $N(\sigma)$ satisfies $P_\sigma$, suppose $\mdl X$ and $(A_n)$ satisfy $Q_\sigma$ and $R_\sigma$. Define $\tilde \sigma$ to be the function
\[
 \tilde \sigma(n) = \left\{
  \begin{array}{ll}
  \sigma_n & \mbox{if $n < \len(\sigma)$} \\
  N(\sigma) & \mbox{otherwise.}
  \end{array}
 \right.
\]
Since $\sigma$ is secured and $\tilde \sigma \supset \sigma$, $M(\tilde \sigma) = M(\hat \sigma) = N(\sigma)$, and $Q_\sigma$ implies
\[
\mu\left(\bigcap_{m \leq N(\sigma)} \bigcup_{n \in [m,\tilde \sigma(m)]} A_n\right) < \lambda'.
\] 
Similarly, $R_\sigma$ implies 
\[
\fa {m \leq N(\sigma)} \mu\left(\bigcup_{n \in [m,N(\sigma)]} A_n \setminus \bigcup_{n \in [m,\tilde\sigma(m)]} A_n\right) \leq (\lambda - \lambda') / 2^{m+1},
\] 
since for $m \geq \len(\sigma)$, $\tilde\sigma(m) = N(\sigma)$. We now use a calculation similar to that of Proposition~\ref{warmup:prop}, with $N(\sigma)$ now playing the role of infinity.
\begin{multline*}
 A_{N(\sigma)} = \bigcap_{m \leq N(\sigma)} \bigcup_{n \in [m,N(\sigma)]} A_n \subseteq  \\
 \left(\bigcap_{m \leq N(\sigma)} \bigcup_{n \in [m,\tilde \sigma(m)]} A_n\right) \cup \bigcup_{m \leq N(\sigma)} \left(\bigcup_{n \in [m,N(\sigma)]} A_n \setminus \bigcup_{n \in [m,\tilde \sigma(m)]} A_n\right).
\end{multline*}
As before, the measure of this set is at most $\lambda' + \sum_{m \leq M} (\lambda - \lambda') / 2^{m+1} < \lambda$, and so $N(\sigma)$ itself satisfies the conclusion of $P_\sigma$. 

In the inductive case where $\sigma$ is not in $T$, we can assume that we have already defined $N(\sigma \concat n)$ for every $n$ so that $P_{\sigma \concat n}$ is satisfied. Define the sequence $n_i$ by setting $n_0 = 0$ and $n_{i+1} = N(\sigma \concat n_i)$, set $\bar m = \len(\sigma)$, and set $N(\sigma) = \max_{i \leq \lceil 2^{\bar m + 1} / (\lambda - \lambda') \rceil} n_i$.

To show that $N(\sigma)$ satisfies $P_\sigma$, fix $\mdl X$ and $(A_n)$ satisfying $Q_\sigma$ and $R_\sigma$. We need to show that there is an $n \leq N(\sigma)$ satisfying $\mu(A_n) < \lambda$. By the definition of $N(\sigma)$, this is the same as showing that for some $i \leq {\lceil 2^{\bar m + 1} / \lambda \rceil}$, there is an $n \leq n_i$ with this property.

Start by trying $i = 1$. Suppose the conclusion fails, that is, there is no $n \leq n_1$ satisfying $\mu(A_n) < \lambda$. Since $n_1 = N(\sigma \concat n_0)$ satisfies $P_{\sigma \concat n_0}$, this implies that either $Q_{\sigma \concat n_0}$  or $R_{\sigma \concat n_0}$ fails. But we are assuming $Q_\sigma$, and that implies $Q_{\sigma \concat n_0}$, so $R_{\sigma \concat n_0}$ fails. This means that there is an $m < \len(\sigma \concat n_0) = \len(\sigma) + 1$ such that 
\[ 
\mu\left(\bigcup_{n \in [m,N(\sigma \concat n_0)]} A_n \setminus \bigcup_{n \in [m,(\sigma \concat n_0)_m]} A_n\right) > (\lambda - \lambda') / 2^{m+1}.
\] 
But our assumption of $R_\sigma$ implies that this does not hold for $m < \len(\sigma)$, since $N(\sigma\concat n_0) = n_1 \leq N(\sigma)$. So the only possibility is that it holds for $m = \bar m = \len(\sigma)$; in other words, we have 
\[
\mu\left(\bigcup_{n \in [\bar m,n_1]} A_n \setminus \bigcup_{n \in [\bar m,n_0]} A_n\right) > (\lambda - \lambda') / 2^{\bar m+1}.
\]

Now repeat this argument for $i = 2, 3, \ldots, \lceil 2^{\bar m + 1} / (\lambda - \lambda') \rceil$. If the conclusion fails each time, then for each $i$ we have
\[
\mu\left(\bigcup_{n \in [\bar m,n_i]} A_n \setminus \bigcup_{n \in [\bar m,n_{i-1}]} A_n\right) > (\lambda - \lambda') / 2^{\bar m + 1}.
\]
This implies $\mu(\bigcup_{n \in [\bar m,N(\sigma)]} A_n) > 1$, a contradiction.
\end{proof}

Notice that the value of $M'$ in the theorem depends on the values of $\lambda$, $\lambda'$, and the \emph{functional} $M$. It is therefore somewhat difficult to make sense of the question as to whether the bound computed in the proof is, in some sense, asymptotically sharp. Given $M$, $\lambda$, $\lambda'$, one \emph{can} effectively determine whether or not a putative value of $M'$ works; so given any bound, one can also compute the least value of $M'$ that satisfies the conclusion. So at issue is not whether we can compute the precise bound, but, rather, come up with a perspicuous characterization of the rate of growth. 

One can easily use recursion along fairly simple trees to define functions that grow astronomically fast. Nonetheless, there are some things we can say about the complexity of $M'$ in terms of $M$. It is well known that G\"odel's system $T$ of primitive recursive functionals of finite type can be stratified into levels $T_n$. At the bottom level, $T_1$, primitive recursion is restricted in such a way that the only \emph{functions} from natural numbers to natural numbers that are definable in the system are primitive recursive. The functionals of $T_1$ are said to be \emph{primitive recursive functionals in the sense of Kleene}, in contrast to the functionals of $T$, which are are said to be \emph{primitive recursive functionals in the sense of G\"odel} (see \cite{kleene:59, avigad:feferman:98,kohlenbach:10}). The results of Howard \cite{howard:81} show the following:

\begin{theorem}
\label{complexity:thm}
 In the previous theorem, if $M$ is definable in G\"odel's $T_n$ for some $n \geq 1$, then, as a function of $\lambda$ and $\lambda'$, $M'$ is definable in $T_{n+1}$.
\end{theorem}

\noindent (See also \cite[Section 10]{kreuzer:kohlenbach:unp}, which relates Howard's results explicitly to the fragments $T_n$.) Theorem~\ref{complexity:thm} implies that if $M$ is a primitive recursive functional in the sense of Kleene, then $M'$ is of level $T_2$ (which is to say, roughly Ackermannian). The results of Kreuzer \cite{kreuzer:unp} provide even more information:

\begin{theorem}
 In Theorem~\ref{combinatorial:thm}, if $M$ is definable in the calculus $\na{G_\infty A^\omega}$ (see, for example, \cite[Section 3]{kohlenbach:08}), then $M'$ is primitive recursive.
\end{theorem}

It would be interesting to know whether these results can be improved. Alternatively, one can consider Theorem~\ref{combinatorial:thm} for particular functionals $M(F)$.  One can show, for example, that with $M(F)=F(0)+n$, the smallest value of $M'$ that works is roughly $n/(\lambda - \lambda')$. Using the algorithm given in the proof of Theorem~\ref{combinatorial:thm} yields the bound $M'=n \cdot \left\lceil 2/(\lambda-\lambda')\right\rceil$, but this can be improved to $n \cdot \left\lceil 1/(\lambda-\lambda')\right\rceil$ by tinkering with the values $(\lambda-\lambda')/2^{m+1}$ in the right hand side of condition $R_\sigma$. An explicit construction gives a lower bound of $n \cdot \left(\left\lceil \frac{1-\lambda'}{\lambda - \lambda'} \right\rceil- 1 \right)$.

However, even for simple functionals like $M(F)=F(F(0))+n$, the combinatorial details quickly become knotty. In this particular case our algorithm gives an $M' = m_{\lceil 2 / (\lambda - \lambda') \rceil}$, where $m_0 = n$ and $m_{i+1} = n \cdot \lceil 2^{m_i+1} / (\lambda - \lambda') \rceil$. This is an iterated exponential in $n$, where the depth of the stack depends on $\lambda - \lambda'$; but we do not know whether such a rate of growth is necessary.

\section{Metastable convergence theorems}
\label{convergence:section}

We can now prove our metastable version of Egorov's theorem.

\begin{theorem}
\label{metastable_egorov}
For every $\varepsilon > 0$, $\lambda > \lambda' > 0$, and functional $M_1(F)$, there is a functional $M_2(F)$ with the following property: for any probability space $\mdl X = (X, \mdl B, \mu)$ and sequence $(f_n)$ of measurable functions, if $M_1(F)$ is a $\lambda'$-uniform bound on the $\varepsilon$-metastable pointwise convergence of $(f_n)$, then $M_2(F)$ is a bound on the $\lambda$-uniform $\varepsilon$-metastable convergence of $(f_n)$. In other words, if for every $F_1$
\[
\mu(\{x \st \ex {m \leq M_1(F_1)} \fa {n, n' \in [m,F_1(m)]} |f_n(x) - f_{n'}(x)| < \varepsilon\}) > 1 - \lambda',  
\]
then for every $F_2$ there is an $m \leq M_2(F_2)$ such that 
\[
\mu(\{x \st \fa {n, n' \in [m,F_2(m)]} |f_n(x) - f_{n'}(x)| < \varepsilon\}) > 1 - \lambda.
\] 
\end{theorem}

\begin{proof}
Fix $\varepsilon > 0$, $\lambda > \lambda' > 0$, and $M_1$. Given $F_2$, define 
\[
M(F) = M_1\left(\lam {m.} \max_{n \in [m,F(m)]} F_2(n)\right),
\] 
and let $M_2(F_2)$ be the value $M'$ given by Theorem~\ref{combinatorial:thm}. Let 
\[
A_n = \left\{ x \st \exists k,k'\in[n,F_{2}(n)]\,|f_{k}(x)-f_{k'}(x)|\geq\varepsilon \right\}.
\]
We wish to show $\mu(A_{n})< \lambda $ for some $m \leq M_2(F_2)$. By the definition of $M_2(F_2)$,  it is enough to show that for every function
$F(m)$, 
\[
\mu\left(\bigcap_{m\leq M(F)}\bigcup_{n\in[m,F(m)]}A_{n}\right)<\lambda'.
\]
For each $m \leq M(F)$, we have 
\begin{align*}
\bigcup_{n\in[m,F(m)]}A_n & =\bigcup_{n\in[m,F(m)]}\bigcup_{k,k'\in[n,F_{2}(n)]}\left\{ x \st |f_{k}(x)-f_{k'}(x)|\geq\varepsilon\right\} \\
 & \subseteq\bigcup_{k,k'\in[m,\max_{n \in [m,F(m)]} F_{2}(n)]}\left\{ x \st |f_{k}(x)-f_{k'}(x)|\geq\varepsilon \right\} 
\end{align*}
Taking $F_{1}(m)=\max_{n \in [m,F(m)]} F_{2}(n)$ in the hypothesis of the theorem gives the desired conclusion. 
\end{proof}

This straightforwardly yields our quantitative version of the dominated convergence theorem.

\begin{theorem}
\label{metastable_dct}	
For every $\varepsilon>0$, $\lambda > \lambda' > 0$, and $M_1(F)$, there is an $M_2(F)$ such that, for any probability space $\mdl X$ and sequence $(f_n)$ of nonnegative measurable functions dominated by the constant function $1$, if $M_1(F)$ is a
$\lambda'$-uniform bound on the $\varepsilon$-metastable pointwise convergence of $(f_n)$, then $M_2(F)$ is a bound on the
$(\varepsilon + \lambda)$-metastable convergence of $(\int f_n)$.
In other words, if for every $F$
\[
\mu(\{x \st \ex {m \leq M_1(F)} \fa {n, n' \in [m,F(m)]} |f_n(x) - f_{n'}(x)| < \varepsilon\}) > 1 - \lambda',  
\]
then for every $F$ there is an $m \leq M_2(F)$ such that 
\[
\fa {n, n' \in [m,F(m)]} \left| \int f_n - \int f_{n'} \right| < \varepsilon + \lambda.
\]
\end{theorem}

\begin{proof}
From the hypotheses, Theorem \ref{metastable_egorov} yields an $M_2(F)$ that is
a bound on the $\lambda$-uniform $\varepsilon$-metastable convergence of $(f_n)$.  Thus, for all $F$, there is
$m\leq M_2(F)$ such that
\[ \mu\left(\{x \mid \fa {n,n'\in [m,F(m)]}
|f_n(x)-f_{n'}(x)| < \varepsilon \}\right) > 1-\lambda. 
\]
Call the set just indicated $A$.
From our choice of $\lambda$ and the definition of $A$, it follows that for all $n,n'\in [m,F(m)]$,
\begin{eqnarray*}
\left|\int f_n - \int f_{n'}\right| & \leq & \int |f_n - f_{n'}| \\
& = &  \int_A |f_n - f_{n'}| + \int_{X \setminus A} |f_n - f_{n'}| \\
& < & \varepsilon + \lambda.
\end{eqnarray*}
That is, $M_2(F)$ provides a bound on the $(\varepsilon + \lambda)$-metastable convergence of $(\int f_n)$ as desired.
\end{proof}

Theorem~\ref{metastable_dct} strengthens Tao's Theorem A.2 \cite{tao:08} in three ways. First, we formulate convergence in terms of the Cauchy criterion, rather than referring to a fixed limit, as Tao does. This is more natural in the context of metastability, and our result implies Tao's, since one can always consider a sequence $f_0, f, f_1, f, f_2, f, \ldots$ in which a fixed limit $f$ has been interleaved. Second, Tao used the stronger hypothesis that $M_1(F)$ provides a bound that works almost everywhere, rather than outside a set of measure at most $\lambda'$. Finally, and most importantly, our proof of Theorem~\ref{combinatorial:thm} provides an explicit description of the bound, $M_2(F)$.

Tao also stated his theorem for the convergence of nets indexed by the directed set $\NN \times \NN$, as was needed in his application. But as he himself noted, the extension to arbitrary countable nets is straightforward. Given any countable net $(f_i)_{i \in I}$, one can define an increasing cofinal sequence $(a_i)_{i \in \NN}$ of elements of the directed set $I$. To adapt Theorem~\ref{combinatorial:thm}, for example, suppose we are given a sequence $(A_n)$ of measurable subsets of a probability space $\mdl X$ with the property that for every function $F$,
\[
\mu\left(\bigcap_{m \leq M(F)} \bigcup_{n \in [a_m,a_{F(m)}]} A_n\right) < \lambda',
\]
where the notation $[a,b]$ denotes $\{ i \st a \leq i \leq b \}$. Define the sequence $(A'_n)_{n \in \NN}$ by $A'_n = \bigcup_{i \in [a_n,a_{n+1}]} A_i$. Then $(A'_n)$ satisfies the requirements of Theorem~\ref{combinatorial:thm}, and hence there is an $i \leq a_{M'}$ such that $\mu(A_i) < \lambda$.

Notice that the expression $\int | f_n - f_{n'}|$ in the proof of Theorem~\ref{metastable_dct} is the $L^1$ norm of $f_n - f_{n'}$. In fact, the same argument shows the following:

\begin{theorem}
\label{Lp_gen}
For every $\varepsilon>0$, $\lambda > \lambda' > 0$, and $M_1(F)$, there is an $M_2(F)$ such that, for any probability space $\mdl X$ and sequence $(f_n)$ of nonnegative measurable functions dominated by the constant function $1$, if $M_1(F)$ is a
$\lambda'$-uniform bound on the $\varepsilon$-metastable pointwise convergence of $(f_n)$, then for every $F_2$ there is an $m \leq M_2(F_2)$ such that for every $p \geq 1$,
\[
\fa {n, n' \in [m,F(m)]} \left\| f_n - f_{n'} \right\|_p < \sqrt[p]{\varepsilon^p + \lambda}.
\]
\end{theorem}

We have considered a metastable version of the dominated convergence theorem where $\mdl X$ is a probability space and the sequence $(f_n)$ is uniformly dominated by the constant function $1$. The dominated convergence theorem itself is usually stated more generally where $\mdl X$ is an arbitrary measure space, and the sequence $(f_n)$ is dominated by an arbitrary integrable function $g$. The general case can be reduced to the one we have considered, taking into account that given an integrable function $g$ and any $\delta_1, \delta_2$ greater than $0$, there is a set $A$ with finite measure such that $\int_{X \setminus A} g < \delta_1$, and a $K$ sufficiently large so that $\int_A (g - \min(g,K)) < \delta_2$. The bound $M_2$ in the conclusion, however, now depends on bounds on $K$ and the size of $A$, for certain $\delta_1$ and $\delta_2$ depending on $\varepsilon$.

\section{A new mode of convergence}
\label{new:convergence:section}

Recall that a sequence $(f_n)$ of measurable functions converges pointwise a.e.~if for almost every $x$,
\begin{equation}
\label{ae:eq}
\fa{\varepsilon > 0} \ex m \fa{n, n' \geq m} |f_n(x) - f_{n'}(x)| < \varepsilon,
\tag{$\mathsf{AE}$}
\end{equation}
and we noted in Section~\ref{introduction:section} it converges almost uniformly if
\begin{equation}
%\label{au:eq}
\fa{\lambda > 0, \varepsilon > 0} \ex m \mu(\{ x \st \ex{n, n' \geq m} |f_n(x) - f_{n'}(x)| \geq \varepsilon \}) < \lambda.
\tag{$\mathsf{AU}$}
\end{equation}
Each of these has an equivalent expression in terms of metastable convergence. Our formulation of Egorov's theorem provides yet another mode of convergence, which we will call \emph{almost uniform metastable pointwise convergence}:
\begin{multline}
\label{aum:eq}
\fa{\lambda > 0, \varepsilon > 0, F} \ex{M}  \\
\mu(\{x \st \fa{m \leq M} \ex {n, n' \in [m,F(m)]} |f_n(x) - f_{n'}(x)| \geq \varepsilon\}) < \lambda.
\tag{$\mathsf{AUM}$}
\end{multline}
In other words, $M$, as a function of $F$, provides a bound on the $\varepsilon$-metastable convergence of the sequences $(f_n(x))$ that is uniform in $x$, and valid outside a set of measure at most $\lambda$. 

Recall that if $\mdl X$ is a probability space, or if the sequence $(f_n)$ is dominated by an $L^p$ function, then a.e.~convergence and almost uniform convergence coincide. More generally, we have the following relationships between these three modes of convergence:

\begin{proposition}
\label{mode:prop}
Let $(f_{n})$ be a sequence of measurable functions on a measure space $\mdl X = (X, \mdl B, \mu)$.
\begin{enumerate}
\item $\mathsf{AU} \rightarrow \mathsf{AUM} \rightarrow \mathsf{AE}$. (Hence, if $\mdl X$ is a probability space or the sequence $(f_n)$ is dominated, the three notions coincide.)
\item If $\mu(\{x:|f_{n}(x)-f_{n'}(x)|\geq\varepsilon\})<\infty$ for all $\varepsilon>0$, $n$, and $n'$, then $\mathsf{AE}$ implies $\mathsf{AUM}$. (In particular, the conclusion holds if for some $p \geq 1$, $f_n \in L^p$ for every $n$.)
\item In general, the implications in \emph{(1)} do not reverse.
\end{enumerate}
\end{proposition}

\begin{proof}
For (1), note that $\mathsf{AU}$ is equivalent to its metastable version, $\mathsf{AU^*}$, which clearly implies $\mathsf{AUM}$. Similarly, $\mathsf{AUM}$ implies, in particular, that for almost every $x$ the sequence $(f_n(x))$ is metastably convergent, and hence convergent.

For (2), prove the contrapositive. Suppose $\mathsf{AUM}$ fails. Then there are $\varepsilon,\lambda,F$
such that for all $M$,
\[
\mu\left(\bigcap_{m\leq M}\bigcup_{n,n'\in[m,F(m)]}\left\{ x \st |f_{n}(x)-f_{n'}(x)|\geq\varepsilon\right\} \right)\geq\lambda.
\]
 By the assumption that each $\left\{ x \st |f_{n}(x)-f_{n'}(x)|\geq\varepsilon\right\} $
has finite measure, we can take the limit as $M\rightarrow\infty$
to get \[
\mu\left(\bigcap_{m}\bigcup_{n\in[m,F(m)]}\left\{ x \st |f_{n}(x)-f_{n'}(x)|\geq\varepsilon\right\} \right)\geq\lambda.
\]
Further, removing $F$ gives,
\[
\mu\left(\bigcap_{m}\bigcup_{n\geq m}\left\{ x \st |f_{n}(x)-f_{n'}(x)|\geq\varepsilon\right\} \right)\geq\lambda.
\]
Hence, $(f_{n})$ is not a.e.~Cauchy.

For (3), $f_{n}=\chi_{[n,n+1]}$ converges $\mathsf{AUM}$ by part (2), but it is easily shown that $f_{n}$ does not converge $\mathsf{AU}$.
Last, $g_{n}:=(-1)^{n}\chi_{[n,\infty)}$ converges
only $\mathsf{AE}$.
\end{proof}

There is also a non-Cauchy version of $\mathsf{AUM}$, which refers to a limit function $f$:
\begin{multline}
\label{aumm:eq}
\ex f \fa{\lambda > 0, \varepsilon > 0, F} \ex{M}  \\
\mu(\{x \st \fa{m \leq M} \ex {n \in [m,F(m)]} |f_n(x) - f(x)| \geq \varepsilon\}) < \lambda.
\tag{$\mathsf{AUM'}$}
\end{multline}

It is easy to see that $\mathsf{AUM'}$ implies $\mathsf{AUM}$, but the converse need not hold; for example, $h_{n}=\chi_{[n,\infty)}$ converges $\mathsf{AUM}$, but not $\mathsf{AUM'}$. Moreover, the analogue of Proposition~\ref{mode:prop} holds when $\mathsf{AUM}$ is replaced by $\mathsf{AUM'}$. Thus we have the following implications, 
\[
\mathsf{AU}\rightarrow\mathsf{AUM'}\rightarrow\mathsf{AUM}\rightarrow\mathsf{AE},
\]
none of which can be reversed in general. 

\section{Final comments}

As noted in Section~\ref{combinatorial:section}, it would be interesting to know the extent to which the bounds we obtain are sharp. For example, can one show that there are functionals $M$ that are primitive recursive in the sense of Kleene for which the $M'$ in Theorem~\ref{combinatorial:thm} is not primitive recursive?

When Tao \cite{tao:06} presented his quantitative version of the dominated convergence theorem, he observed that the bound $M'$ can be computed in principle.
\begin{quote}
 In practice, though, it seems remarkably hard to do; the proof of the Lebesgue dominated convergence theorem, if inspected carefully, relies implicitly on the infinite pigeonhole principle, which is notoriously hard to finitize.
\end{quote}
He went on to note that since the Lebesgue dominated convergence theorem is equivalent, in the sense of reverse mathematics, to the arithmetic comprehension axiom $\ax{ACA}$ \cite{yu:94}, the dependence of $M'$ on the parameters is likely to be ``fantastically poor.'' The dependence we have obtained is, indeed, rather poor, but it is at least explicit and comprehensible.

In fact, the axiomatic strength of the dominated convergence theorem is sensitive to the way in which it is formulated.
Elsewhere \cite{avigad:dean:rute:unp} we have shown that the formulation of the dominated convergence theorem that corresponds to Tao's quantitative version is strictly weaker than $\ax{ACA}$. It is possible, however, that the quantitative version, which quantifies over continuous functionals, is axiomatically stronger than the original. In fact, we suspect that each of Theorem~\ref{combinatorial:thm}, \ref{metastable_egorov}, and \ref{metastable_dct} is equivalent to $\ax{ACA}$. This is reminiscent of Gaspar and Kohlenbach~\cite{gaspar:kohlenbach:10}, which provides a sense in which a quantitative version of the infinitary pigeonhole principle is axiomatically stronger than the non-quantitative version.

The results here can be viewed as instances of ``proof mining,'' which aims to extract quantitative and computationally meaningful information from nonconstructive results in analysis; see \cite{kohlenbach:08} and \cite{avigad:et:al:10,kohlenbach:leustean:09,kohlenbach:10,kohlenbach:unp:e}. In particular, the passage from Proposition~\ref{warmup:prop} to Theorem~\ref{combinatorial:thm} can be seen as an instance of the general method of eliminating a choice principle in favor of bar recursion, described in \cite[Section 11.3]{kohlenbach:08}. We are grateful to Paulo Oliva for pointing this out to us.

% \bibliographystyle{plain} 
% \bibliography{master}

\end{document}